\documentclass[oneside,english]{amsart}
\usepackage{amsmath,amssymb,amsmath,amscd,amsfonts,amsthm,mathrsfs, graphicx}
\usepackage[arrow,matrix,curve,cmtip,ps]{xy}
\usepackage{enumitem}
\usepackage{pinlabel}
\usepackage{tikz-cd}
\usepackage{mathtools}
\usepackage{color}
\usepackage{xcolor}
\usepackage{caption}
\usepackage{subcaption}
\usepackage{soul}
\usepackage{svg}
\usepackage{url}

\pdfoutput=1
\usepackage{geometry}
\usepackage{cancel}

\usepackage{setspace}

\usepackage{relsize}

\usepackage{centernot}

\usepackage[pagebackref=true]{hyperref}
\hypersetup{colorlinks,linkcolor={blue},citecolor={black}}

\usepackage{indentfirst}

\usepackage{pinlabel}	

\newtheorem{proposition}{Proposition}[section]
\newtheorem{theorem}[proposition]{Theorem}
\newtheorem{corollary}[proposition]{Corollary}
\newtheorem{lemma}[proposition]{Lemma}

\newtheorem*{theorem*}{Theorem}
\newtheorem*{proposition*}{Proposition}
\newtheorem*{lemma*}{Lemma}
\newtheorem*{corollary*}{Corollary}
\newtheorem{question*}{Question}
\newtheorem*{rep@theorem}{\rep@title}
\newcommand{\newreptheorem}[2]{
\newenvironment{rep#1}[1]{
 \def\rep@title{#2 \ref{##1}}
 \begin{rep@theorem}}
 {\end{rep@theorem}}}
\makeatother

\newtheorem*{rep@proposition}{\rep@title}
\newcommand{\newrepproposition}[2]{
\newenvironment{rep#1}[1]{
 \def\rep@title{#2 \ref{##1}}
 \begin{rep@proposition}}
 {\end{rep@proposition}}}
\makeatother

\theoremstyle{definition}
\newtheorem{definition}[proposition]{Definition}

\newtheorem{remark}[proposition]{Remark}

\newreptheorem{theorem}{Theorem}
\newreptheorem{lemma}{Lemma}
\newreptheorem{proposition}{Proposition}
\newreptheorem{corollary}{Corollary}
\newreptheorem{definition}{Definition}

\newcommand{\bdry}{\partial}

\newcommand{\Q}{\mathbb{Q}}

\newcommand{\Z}{\mathbb{Z}}
\newcommand{\R}{\mathbb{R}}

\newcommand{\C}{\mathcal{C}}

\newcommand{\K}{\mathcal{K}}
\newcommand{\A}{\mathcal{A}}

\newcommand{\AC}{\mathcal{AC}}

\newcommand{\Hom}{\operatorname{Hom}}

\newcommand{\rk}{\operatorname{rk}}
\newcommand{\Bl}{\mathcal{B}\ell}
\newcommand{\Qt}{{\Q[t^{\pm 1}]}}
\newcommand{\Zt}{{\Z[t^{\pm 1}]}}

\newcommand{\into}{\hookrightarrow}

\begin{document}
\title[Obstructing Two-torsion in the Rational Knot Concordance Group]{Obstructing Two-torsion in the Rational Knot Concordance Group}

\author{Jaewon Lee}
\address{Department of Mathematical Sciences, Korea Advanced Institute for Science and Technology}
\email{freejw@kaist.ac.kr}

\date{\today}

\def\subjclassname{\textup{2020} Mathematics Subject Classification}
\expandafter\let\csname subjclassname@1991\endcsname=\subjclassname

\begin{abstract} 
It is well known that there are many 2-torsion elements in the classical knot concordance group. On the other hand, it is not known if there is any torsion element in the rational knot concordance group $\C_\Q$. Cha defined the algebraic rational concordance group $\AC_\Q$, an analogue of the classical algebraic concordance group, and showed that $\AC_\Q\cong\Z^\infty\oplus\Z_2^\infty\oplus\Z_4^\infty$. The knots that represent 2-torsions in $\AC_\Q$ potentially have order $2$ in $\C_\Q$. In this paper, we provide an obstruction for knots of order $2$ in $\AC_\Q$ from being of finite order in $\C_\Q$. Moreover, we give a family consisting of such knots that generates an infinite rank subgroup of $\C_\Q$. We also note that Cha proved that in higher dimensions, the algebraic rational concordance order is the same as the rational knot concordance order. Our obstruction is based on the localized von Neumann $\rho$-invariant.
\end{abstract}

\maketitle
\section{Introduction}

Two knots in $S^3$ are said to be \textit{concordant} if they cobound a properly embedded locally flat annulus in $S^3\times I$. The concordance classes form an abelian group $\C$ under connected sum, called the \textit{knot concordance group}. The identity element of $\C$ is the concordance class of knots bounding an embedded locally flat disk in $B^4$, called \textit{slice knots}. Since the inverse for a knot $K$ in $\C$ is represented by its reversed mirror image $r\overline{K}$, any non-slice knot $K$ which is negative amphichiral, i.e., $K$ is isotopic to $r\overline{K}$, represents a 2-torsion element in $\C$. There are many such knots like the figure-eight knot, and in fact, we know that $\C$ contains $\Z_2^\infty$ as a subgroup. We make a remark that Gordon \cite[Problem 1.94]{Kir97} conjectured that any 2-torsion element in $\C$ would be concordant to a negative amphichiral knot. Naturally, we can ask if there are any other torsions in $\C$. 

\begin{question*}\cite[Problem 1.32]{Kir97}
  Does $\C$ contain a torsion element of order other than $2$? In particular, does $\C$ contain an order $4$ element?
\end{question*}

A good candidate comes from the \textit{algebraic concordance group} $\AC$. Recall that Levine \cite{Lev69} defined the following surjective homomorphism:
\begin{equation*}
  \phi:\C\twoheadrightarrow \AC
\end{equation*}
and showed that $\AC$ is isomorphic to $\Z^\infty \oplus \Z_2^\infty \oplus \Z_4^\infty$.
Moreover, he proved that for odd higher-dimension, $\phi$ is an isomorphism.

A knot is said to be \textit{algebraically slice} if its image in $\AC$ is trivial. In contrast to higher dimensions, $\ker\phi$ turned out to be nontrivial in the classical dimension, as proved by Casson and Gordon \cite{CG78, CG86}. The provided examples were shown by Jiang \cite{Jia81} to generate an infinite rank subgroup in $\ker\phi\le \C$. An element in $\C$ whose image under $\phi$ has order $4$ is a potential candidate for a 4-torsion in $\C$. Livingston and Naik \cite{LN99, LN01} (see also \cite{JN07, GL21}) proved that certain knots having order $4$ in $\AC$ are of infinite order in $\C$.

In this article, we consider an analogous question for the rational knot concordance group. A knot in $S^3$ is called \textit{rationally slice} if it bounds a locally flat embedded disk in some rational homology $4$-ball. The \textit{rational knot concordance group}, or briefly the rational concordance group, $\C_\Q$ is formed by the quotient of $\C$ modulo rationally slice knots. A negative amphichiral knot is called \textit{strongly negative amphichiral} if the isotopy from $K$ to $r\overline{K}$ is an involution on $S^3$. The figure-eight knot is the simplest nontrivial example and was proved to be rationally slice by Cochran, based on \cite{FS84}. Kawauchi \cite{Kaw09} proved that any strongly negative amphichiral knot is rationally slice.\footnote{Such a knot is smoothly rationally slice, i.e., it bounds a smoothly embedded disk in a rational ball.} Kim and Wu \cite{KW18} showed that negative amphichiral knots satisfying certain properties are (smoothly) rationally slice. It is not known if there exists a negative amphichiral knot which is not rationally slice. This motivates the following question:
\begin{question*}
  Does $\C_\Q$ contain any torsion element? In particular, does $\C_\Q$ contain an order $2$ element?
\end{question*}

Again, a good candidate for 2-torsion comes from the \textit{algebraic rational concordance group} $\AC_\Q$ defined by Cha \cite{Cha07}. He also defined the surjective homomorphism $\phi_\Q:\C_\Q\twoheadrightarrow\AC_\Q$ and proved that 
\begin{equation*}
  \AC_\Q\cong  \Z^\infty \oplus \Z_2^\infty \oplus \Z_4^\infty.
\end{equation*}
We remark that Cha also proved that $\phi_\Q:\C_\Q\twoheadrightarrow\AC_\Q$ is an isomorphism in higher dimensions. As one can observe in the following commutative diagram:
\begin{center}
\begin{tikzcd}
&\C \ar[twoheadrightarrow]{r}{\psi}\ar[twoheadrightarrow]{d}{\phi} &\C_\Q \ar[twoheadrightarrow]{d}{\phi_\Q}\\
&\AC \ar[twoheadrightarrow]{r}{\overline{\psi}} &\AC_\Q
\end{tikzcd}
\end{center}
the knots that represent 2-torsions in $\AC_\Q$ potentially have order $2$ in $\C_\Q$. We obstruct some of them from being actually torsion elements.
\begin{theorem}\label{main A}
There are knots of order $2$ in $\AC_\Q$ which generate $\Z^\infty$ in $\C_\Q$.
\end{theorem}
\noindent  Theorem \ref{main A} implies a difference between higher dimensional rational concordance groups and the rational concordance group in classical dimension.

To prove Theorem \ref{main A}, we construct a rational slice obstruction by employing the von Neumann $\rho$-invariant. It was developed as a representation invariant of 3-manifolds  \cite{ChG85}. Cochran, Orr, and Teichner \cite{COT03} introduced the $\rho$-invariant of the $0$-surgery of a knot as a knot concordance invariant to find a filtration structure, called \textit{solvable filtration}, on $\C$. Cochran, Harvey, and Leidy \cite{CHL09} defined the \textit{first order signature} $\rho^{(1)}(K)$ for a knot $K$, based on the von Neumann $\rho$-invariant, and used it to prove that each layer of the filtration contains an infinite rank subgroup (see also \cite{COT04, Cha07, Har08, CHL11a, CHL11b, Dav14, Kim20, DPR21}). Using the first order signature, our obstruction can be summarized as follows:
\begin{theorem}\label{main B}
  Let $\{K_i\}$ be a set of knots whose order are finite in $\AC_\Q$ and $\Delta_i$ be the Alexander polynomial of $K_i$. If $\Delta_i$'s are strongly irreducible and pairwisely strongly coprime, and $\rho^{(1)}(K_i)\neq 0$, then $\{K_i\}$ is linearly independent in $\C_\Q$.
\end{theorem}
\noindent Notice that the finite order condition in Theorem \ref{main B} does not include the algebraically rationally slice cases because of strong irreducibility of $\Delta_i$, by generalized Fox-Milnor condition \cite{Cha07, CFHH13}.

\begin{figure}[h]
\hspace{\fill}
\includesvg{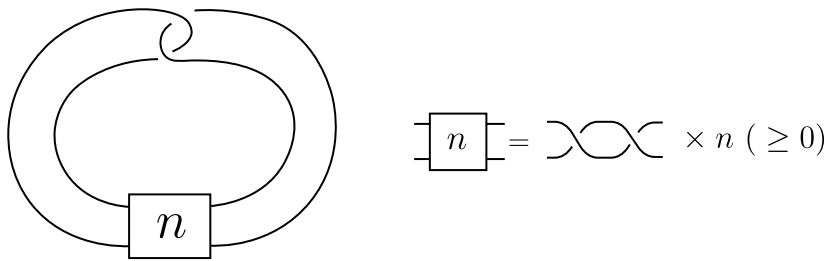}
\caption{$n$-twist knot with a positive clasp.}
\label{twist knot}
\end{figure}

We obtain an infinite family of knots satisfying the condition in Theorem \ref{main B}, thus Theorem \ref{main A} follows. Such knots of finite order in $\AC_\Q$ arise from the $n$-twist knots $K_n$ with nonnegative $n$. The $n$-twist knot is the positive Whitehead double of the unknot with $n$ twists as illustrated in the Figure \ref{twist knot}. Levine \cite{Lev69} classified their algebraic concordance order. Bullock and Davis \cite{BD12} proved that whenever $n$ is not a perfect power and $K_n$ is not algebraically slice, the Alexander polynomials $\Delta_n$ of $K_n$ are strongly irreducible. They also showed that the $\Delta_n$'s are pairwisely strongly coprime. Moreover, Davis \cite{Dav12a, Dav12b} proved that the first order signature $\rho^{(1)}(K_n)$ does not vanish for infinitely many $K_n$ of order $2$ in $\AC$.

\begin{theorem}\label{main C}
$n$-twist knots with order 2 in $\AC_\Q$ generate an infinite rank subgroup in $\C_\Q$.
\end{theorem}

It was previously shown by Davis \cite{Dav12a, Dav12b} that $K_n$ which are of order 2 in $\AC$ generate an infinite rank subgroup in $\C$. (See also \cite{Kim05, Ilt22}.) As one can see, the Theorem \ref{main B} covers the case for order $4$ in $\AC_\Q$, but our examples only involve order 2 cases due to the general difficulty in computing $\rho^{(1)}$. We do not know if there is a knot of order $4$ in $\AC_\Q$ that satisfies the conditions of Theorem \ref{main B}.

We make a remark on our twist knot $K_n$. Recall that there are several tools to obstruct rational sliceness. For example, the knot signature and the Levine-Tristram signature function vanish for the rationally slice knots. \cite{CO93, CK02}. Moreover, many Heegaard Floer theoretic invariants such as $\tau$ \cite{OS03, Ras03}, $\epsilon$ \cite{Hom14}, $\nu^+$ \cite{HW16}, and $\Upsilon$ \cite{OSS17} also vanish for the rationally slice knots, as the $1$-surgery of a rationally slice knot bounds a rational ball.

For our $n$-twist knots $K_n$ in target, however, all the aforementioned invariants vanish. \footnote{One can take our $K_n$ with infection by a knot whose $\tau$ is non-zero and use $\Upsilon$ to prove Theorem \ref{main A} in smooth category.} This is because the $n$-twist knots $K_n$ with $n\ge 0$ are $0$-bipolar \cite{CHH13}, and in particular, they are BPH-slice \cite{MP23}. In other words,
\begin{equation*}
  \sigma (K_n) = \sigma_{K_n}(\omega) = \tau(K_n) = \epsilon (K_n) = \nu^+(K_n) = \Upsilon_{K_n} (t) = 0.
\end{equation*}

On the other hand, we introduce the notion of complexity \cite{COT03, Cha07} on the von Neumann $\rho$-invariant to prove Theorem \ref{main C}. The complexity of a rationally slice knot is, roughly speaking, the positive integer $c$ counting with sign how many times meridian meets the slice disk. We can introduce complexity $c$ on Alexander module, denoted as $\A_c$, and define corresponding von Neumann $\rho$-invariant $\rho_c$ to obstruct a knot from being rationally slice with complexity $c$. Since this invariant is defined for each complexity $c$, as a rational slice obstruction, it is necessary to verify that all values $\rho_c$ do not vanish. However, localization on the rational Alexander module at a polynomial $p$, denoted as $\A_{c, p}$, can provide a single value that does not depend on $c$. The main ingredient of our construction involves the localization on Alexander module with complexity and its corresponding $\rho_{c, p}$-invariant.

Note that a knot $K$ is rationally slice with complexity $c$ if and only if the $(c, 1)$-cable $K_{c, 1}$ is strongly rationally slice, i.e., rationally slice with complexity $1$ \cite{CFHH13}. Thus, our $\rho_{c, p}$ for $K$ gives an obstruction for $K_{c, 1}$ from being strongly rationally slice. Moreover,

\begin{corollary}\label{main D}
  Let $K_n$ be an $n$-twist knot with order $2$ in $\AC_\Q$. If $\rho^{(1)}(K_n)\neq 0$, then any cable $(K_n)_{c, 1}$ has infinite order in $\C$.
\end{corollary}

On the other hand, $K_n$ when $n$ is a square is algebraically rationally slice, and moreover, $K_1$, which is the figure-eight knot, is rationally slice. It is not known if $(K_1)_{c, 1}$ has infinite order in $\C$. In contrast, in the smooth knot concordance group, it is known that $(K_1)_{c, 1}$ has infinite order for odd $c$ \cite{HKPS22}, for $c=2$ \cite{DKMPS22} (see also \cite{ACMPS23}), and for $c=2n$ when $n$ is odd \cite{KPT24}. In particular, such knots are examples of infinite order rationally slice knots in the smooth category, whereas it is not known if such a knot exists in the topological category. We close the introduction with the following analogous question in the topological category:
\begin{question*}\label{q3}
  Does there exist a rationally slice knot of infinite order in $\C$?
\end{question*}

\noindent\textbf{Organization} In Section 2, we recall definition of von Neumann $\rho$-invariant and its basic properties. After that, we introduce localization and complexity separately on Alexander module and define the corresponding $\rho$-invariants. In Section 3, we define the localized $\rho$-invariant with complexity and check its properties including additivity. In Section 4, we give a constraint for a knot and show that the invariant vanishes for rationally slice knots satisfying the constraint. Then we prove Theorem \ref{main B}. In Section 5, we prove Theorem \ref{main C} by applying Theorem \ref{main B}. We also prove \ref{main D} and ask some questions about rational concordance order of certain knots and concordance order of their cable. In Section 6, we briefly recall the notions of 0-bipolar and BPH-slice, and discuss some relations of 0-bipolar, BPH-slice, and rational slice.\\

\noindent\textbf{Acknowledgement} The author is grateful to his advisor JungHwan Park for generous advice. The author would also like to thank Jae Choon Cha and Taehee Kim for helpful comments. This work is partially supported by Samsung Science and Technology Foundation (SSTF-BA2102-02) and the POSCO TJ Park Science Fellowship.

\section{Background}\label{s-2}
In this section, we first review the definition of von Neumann $\rho$-invariant by using $L^2$-signature defect and see its basic properties. Then we briefly recall some ordinary concordance invariants based on the $\rho$-invariant such as the zeroth and first order signatures. After that, we explain the localization of the Alexander module, the extended Blanchfield form, and the corresponding $\rho$-invariant. At last, we recall the notion of complexity, introduce it on Alexander module and Blanchfield form, and define a rational version of the first signatures. Notice that the last two subsections will be the main ingredients of our construction. We claim no own result in this section. Throughout this section, every manifold is assumed to be compact and oriented, unless otherwise stated.

\subsection{von Neumann $\rho$-invariant}
A von Neumann $\rho$-invariant is an invariant of a 3-manifold $M$ and its representation $\phi:\pi_1(M)\rightarrow \Gamma$ for a discrete group $\Gamma$. It can be obtained by the signature of the bordism class of $(M, \phi)$. Hence, the $L^2$-signature defect (the difference between the $L^2$-signature and the ordinary signature) of a 4-manifold $W$ bounded by $M$ with extension $\psi:\pi_1(W)\rightarrow \Gamma$ of $\phi$, can be used to define von Neumann $\rho$-invariant. For the definition of $L^2$-signature $\sigma^{(2)}$, see \cite{Luc02}, \cite[Section 5]{COT03}, and \cite{Cha08}.

\begin{lemma}\cite[Appendix]{CW03}, \cite[below of Lemma 2.7]{Cha08}\label{existence}
  For a closed, connected 3-manifold $M$ with a representation $\phi:\pi_1(M)\rightarrow \Gamma$, there exists a 4-manifold $W$ bounded by $k$ copies $M_i$ of $M$ for some $k$ and $\psi:\pi_1(W)\rightarrow \Lambda$ for a group $\Lambda$ of which $\Gamma$ is a subgroup such that below diagram commutes for each $i=1,\cdots, k$.
\begin{center}
\begin{tikzcd}
\pi_1(M_i)\ar{r}{\phi}\ar{d} &\Gamma \ar[hook]{r} &\Lambda\\
\pi_1(W)\ar[swap]{urr}{\psi}
\end{tikzcd}
\end{center}
\end{lemma}

\begin{definition}
The \textit{von Neumann $\rho$-invariant} for above $(M, \phi)$ is defined by
\begin{equation*}
  \rho(M, \phi) = \frac{1}{k}\left(\sigma^{(2)}(W, \psi) - \sigma(W)\right),
\end{equation*}
where $k$ and $(W, \psi)$ are given in Lemma \ref{existence}.
\end{definition}
\noindent This value is independent of the choice of $(W, \psi)$ by \cite[Lemma 3.6]{Har08} and the following lemma, called \textit{subgroup property}:

\begin{proposition}\cite[Proposition 5.13]{COT03} Given $\phi:\pi_1(M)\rightarrow \Gamma$, if $j:\Gamma\rightarrow \Lambda$ is injective, then
  \begin{equation*}
    \rho(M,j\circ \phi) = \rho(M, \phi).
  \end{equation*}
\end{proposition}

\noindent The value defined above coincides with the originally defined von Neumann $\rho$-invariant by Cheeger and Gromov \cite{ChG85}, and the following lemma follows.
\begin{lemma}\cite[Remark 5.10]{COT03}\label{rho-sg}
Suppose that closed connected 3-manifolds $M_1, \cdots, M_n$ cobound a 4-manifold $W$ and the following diagram commutes for each $i$:
\begin{center}
\begin{tikzcd}
\pi_1(M_i)\ar{r}{\phi_i}\ar{d} &\Gamma_i \ar{d}{j_i}\\
\pi_1(W)\ar{r}[hook]{\psi}&\Lambda
\end{tikzcd}
\end{center}
If each $j_i$ is injective, then we obtain
\begin{equation*}
  \sum\limits_{i=1}^{n}\rho(M_i, \phi_i) = \sigma^{(2)}(W, \psi) - \sigma(W).
\end{equation*}
\end{lemma}

In this paper, we mainly have an interest in a cobordism between 3-manifolds which may not be homeomorphic. Thus, Lemma \ref{rho-sg}, combined with the following lemma, will be frequently used to compute a $\rho$-invariant of 3-manifold. Below proposition immediately comes from the trivial cobordism $M\times I$ by Lemma \ref{rho-sg}.

\begin{proposition}\label{orientation-reversing}
Let $-M$ be the orientation-reversed one of $M$. By abusing of notation, we still use $\phi$ for the induced representation of $-M$. Then,
\begin{equation*}
  \rho(-M, \phi) = -\rho(M, \phi).
\end{equation*}
\end{proposition}

\noindent We close this subsection by introducing a useful lemma to compute $\sigma^{(2)}(W, \phi)$ when $\Gamma$ is a poly-torsion-free-abelian (PTFA) group. Note that all concerned group for $\rho$-invariant in this paper are PTFA. For the definition of PTFA group, see \cite[Definition 2.1]{COT03}. 

\begin{lemma}\label{sg-bound}
Let $\Gamma$ be a PTFA group, and $\K$ be the fraction field of the group ring $\Q\Gamma$. Then
\begin{equation*}
  |\sigma^{(2)}(W,\psi:\pi_1(W)\rightarrow \Gamma)|\le \rk_{\K}H_2(W;\K).
\end{equation*}
In particular, if $H_2(W;\K)=0$, then the $L^2$-signature $\sigma^{(2)}(W, \psi)$ vanishes.
\end{lemma}
\begin{proof}
  By definition of $L^2$-signature \cite{Luc02, COT03}, $|\sigma^{(2)}(W,\psi)|\le\rk_{\Q\Gamma} H_2(W;\Q\Gamma)$. Since $\rk_{\Q\Gamma}H_2(W;\Q\Gamma)=\rk_{\K}H_2(W;\K)$ (for example, see \cite[Lemma 2.4]{Cha08}), the result holds.
\end{proof}

\subsection{Knot invariants based on $\rho$}
The von Neumann $\rho$-invariant of the 0-surgery $M_K$ of a knot $K$ with a specified representation can be adopted as a knot concordance invariant. In this subsection, we recall several invariants defined in such a way.
\begin{definition}\cite[Definition 3.1]{CHL10}
The \textit{zeroth order signature} $\rho^{(0)}(K)$ of a knot $K$ is defined by $\rho(M_K, \phi_0)$ where $\phi_0:\pi_1(M_K)\rightarrow \Z$ sends a meridian to $1$.
\end{definition}
\noindent Note that $\phi$ is just the abelianization, so above $\Z$ is indeed the first homology $H_1(M_K;\Z)$ of $M_K$. Although we will not compute this invariant directly in this paper, we present how this \textit{abelian} invariant can be related with the Levine-Tristram signature function. For simplicity, we normalize $\int_{S^1}d\omega = 1$.
\begin{theorem}\label{thm: zeroth}\cite{COT04}
Let $\sigma_K$ be the Levine-Tristram signature function for a knot $K$. Then,
\begin{equation*}
  \rho^{(0)}(K) = \int_{S^1}\sigma_K(\omega) d\omega.
\end{equation*}
In particular, if $K$ is algebraically slice, then $\rho^{(0)}(K)=0$.
\end{theorem}
To distinguish non-slice knots among the algebraically slice knots, we need a finer obstruction. Notice that the PTFA group $\Z$, say $\Gamma_0$, used to define $\rho^{(0)}(K)$ is
\begin{equation*}
  \Gamma_0 = \frac{\pi_1(M_K)}{\pi_1(M_K)^{(1)}}.
\end{equation*}
Thus, $\rho^{(0)}(K)$ captures homological information of the infinite cyclic cover of $M_K$, which can be interpreted as
\begin{equation*}
  H_1(M_K;\Z\Gamma_0)\cong \frac{\pi_1(M_K)^{(1)}}{\pi_1(M_K)^{(2)}}.
\end{equation*}
To catch information lying deeper inside such as $\pi_1(M_K)^{(2)}$, one may construct a larger PTFA group $\Gamma_1$ and see the corresponding higher cover. Before going on, we recall the definition of Blanchfield form. From now on, we denote the Alexander module $H_1(M_K;\Qt)$ by $\A(K)$. To avoid confusion, note that $\A(K)$ is induced by the trivial inclusion $\Zt \hookrightarrow \Qt$. Since $\A(K)$ is a torsion $\Qt$-module, the linking form is well-defined.

\begin{definition}\label{def: Bl} \cite{COT03}
The \textit{Blanchfield form} $\Bl:\A(K)\times \A(K)\rightarrow \Q(t)/\Qt$ is defined to make the following diagram commutes:
\begin{center}
\begin{tikzcd}
\A(K)\ar{d}{\cong}[swap]{PD}\ar{r}{x\mapsto\Bl(x,\cdot)} &\Hom_{\Qt}\left(\A(K), \frac{\Q(t)}{\Qt}\right)\\
H^2(M_K;\Qt)\ar{r}[swap]{\cong}{B^{-1}}&H^1\left(M_K;\frac{\Q(t)}{\Qt}\right)\ar{u}{\cong}[swap]{\kappa}
\end{tikzcd}
\end{center}
$PD$ is the Poincaré duality, $B^{-1}$ is the inverse of Bockstein homomorphism, and the $\kappa$ is the Kronecker evaluation map in the universal coefficient spectral sequence. $B$ is isomorphism since $H^2(M_K;\Qt)$ is torsion, and the isomorphism of $\kappa$ comes from the universal coefficient spectral sequence theorem \cite{Lev77}.
\end{definition}

We say a submodule $P\subset\A(K)$ is \textit{isotropic} if $P\subset P^\perp$ with respect to $\Bl$. If $P=P^\perp$, it is called \textit{Lagrangian}.

\begin{definition}\cite[Definition 4.1]{CHL10}
Let $P$ be an isotropic submodule of $\A(K)$. \textit{A first order} $L^2$-\textit{signature} $\rho^{(1)}(K, P)$ is defined by $\rho(M_K,\phi_P)$ where $\phi_P$ is the quotient map by the kernel of:
\begin{equation*}
  \pi_1(M_K)^{(1)}\rightarrow \frac{\pi_1(M_K)^{(1)}}{\pi_1(M_K)^{(2)}}\hookrightarrow \A(K) \rightarrow \A(K)/P.
\end{equation*}
In particular, when $P=0$, by subgroup property,
\begin{equation*}
  \rho(M_K, \phi_P) = \rho\left(M_K,\phi_1:\pi_1(M_K)\rightarrow \frac{\pi_1(M_K)}{\pi_1(M_K)^{(2)}}\right).
\end{equation*}
For this case, we call it by \textit{the first order signature} and write it as $\rho^{(1)}(K)$.
\end{definition}
\noindent Observe that the following sequence is exact.
\begin{equation*}
  0\rightarrow \frac{\pi_1(M_K)^{(1)}}{\pi_1(M_K)^{(2)}} \rightarrow \frac{\pi_1(M_K)}{\pi_1(M_K)^{(2)}} \rightarrow \frac{\pi_1(M_K)}{\pi_1(M_K)^{(1)}} \rightarrow 0
\end{equation*}
Since $\frac{\pi_1(M_K)^{(1)}}{\pi_1(M_K)^{(2)}}$ can be identified with the integral Alexander module $\A^\Z=H_1(M_K;\Zt)$ and $\frac{\pi_1(M_K)}{\pi_1(M_K)^{(1)}}\cong \Z$, our $\Gamma_1=\frac{\pi_1(M_K)}{\pi_1(M_K)^{(2)}}$ can be regarded as $\A^\Z\rtimes \Z$. 

A first order $L^2$-signature is not a concordance invariant by itself, but the set of such values can work as a slice obstruction.
\begin{theorem}\cite[Theorem 4.2, 4.4]{COT03} \cite[Theorem 4.2]{CHL10}\label{CHL-derivative}
If $K$ is slice, then there exists a Lagrangian submodule $P$ of $\A(K)$ such that $\rho^{(1)}(K, P)=0$. For a slice disk $\Delta$, above $P$ is given by
\begin{equation*}
  P = \ker(H_1(M_K;\Qt)\rightarrow H_1(B^4-\Delta;\Qt)).
\end{equation*}
\end{theorem}

\begin{theorem}\cite[Proposition 5.8]{CHL09}\label{CHL-Q}
  If $K$ is rationally slice, then one of the first order $L^2$-signatures of $K$ vanishes. In other words, there exists an isotropic submodule $P$ of $\A(K)$ such that $\rho^{(1)}(K, P)=0$.
\end{theorem}

\subsection{Localized $\rho$-invariant}\label{ss-2.3} From the perspective of obstructing a knot $K$ from being slice, it is required to check that for all isotropic submodule $P\subset \A$, $\rho^{(1)}(K, P)\neq 0$. Instead of looking through all such a submodule, one strategy is to impose a strong condition for Alexander module $\A(K)$ and check a single $\rho$-invariant. A localization technique fits well for this strategy. This subsection is based on the work by Davis \cite{Dav12b}. We assume that every polynomial $p$ be a symmetric Laurent polynomial with rational coefficients.

\begin{definition}
Let $R_p=\left\{\frac{f}{g}\in\Q(t): (p, g)=1, f, g\in \Qt\right\}$. The \textit{localized Alexander module at} $p$ is
\begin{equation*}
  \A_p(K) = \A(K)\otimes R_p,
\end{equation*}
where the $\Qt$-tensor product is the usual multiplication.
\end{definition}

\begin{definition}
Let $\pi_1(M_K)^{(2)}_p$ be the kernel of the following map:
\begin{equation*}
  \pi_1(M_K)^{(1)}\rightarrow\frac{\pi_1(M_K)^{(1)}}{\pi_1(M_K)^{(2)}}=H_1(M_K; \Z[t^{\pm 1}])\into \A(K)\rightarrow \A(K)\otimes R_p=\A_p(K).
\end{equation*}
The \textit{localized $\rho$-invariant of $K$ at $p$} is defined as:
\begin{equation*}
  \rho_p^{(1)}(K) = \rho\left(M_K, \phi_p:\pi_1(M_K)\rightarrow \frac{\pi_1(M_K)}{\pi_1(M_K)^{(2)}_p}\right).
\end{equation*}
\end{definition}

\noindent After localization at $p$, any submodule of $\A(K)$ whose order is relatively prime with $p$ is annihilated in $\A_p(K)$. Then we may expect some isotropic submodules to be killed. Since $p$ is assumed to be symmetric, the Blanchfield form $\Bl$ naturally extends to $\Bl_p$ on $\A_p(K)$ as
\begin{equation*}
  \Bl_p(x\otimes f(t), y\otimes g(t)) = f(t)\cdot \Bl(x, y)\cdot g(t^{-1}).
\end{equation*}

\begin{definition}
A knot $K$ is \textit{$p$-anisotropic} if $\A_p(K)$ has no nontrivial isotropic submodule with respect to $\Bl_p$.
\end{definition}
\noindent Now $p$-anisotropy condition for $K$ makes $\rho_p^{(1)}(K)$ be enough to obstruct $K$ from being slice. Moreover,
\begin{theorem}\cite[Theorem 4.1]{Dav12b}\label{loc obs}
  Let $K_1, \cdots, K_n$ be $p$-anisotropic knots. If $K=K_1\# \cdots \# K_n$ is slice, then $\rho_p^{(1)}(K)=0$.
\end{theorem}

\noindent Note that even if each $K_i$ is $p$-anisotropic, the connected sum $K_1\#\cdots \# K_n$ may not be $p$-anisotropic.

The vanishing of the signature defect of the slice disk complement bounded by $M_K$ is anticipated. The essence of localization and $p$-anisotropy is that we expect all possible Lagrangian (or generally isotropic) submodule of $\A(K)$ to be killed in $\A_p(K)$. The key role of $p$-anisotropy is to make $\A_p(K)$ inject to $\A_p(W)= H_1(W;\Qt)\otimes R_p$ so that $\rho^{(1)}_p(K)$ can be well-defined as the signature defect of $W$. Then it works well as a single-valued concordance invariant, without computing all $\rho^{(1)}(K, P)$ for all isotropic submodule $P$. One trade-off is that $\rho_p^{(1)}(K)$ can only be applied to the subgroup generated by the $p$-anisotropic knots. Before moving on to the next subsection, we give a sufficient condition for $p$-anisotropy.
\begin{proposition}\label{p-anisotropy condition}
If each factor of $p$ is symmetric and divides the Alexander polynomial of $K$ with multiplicity at most 1, then $K$ is $p$-anisotropic.
\end{proposition}

\subsection{$\rho$-invariant with complexity}\label{ss-2.4} Now suppose that $K$ is slice in a rational homology ball $V$ via a disk $\Delta^2$. We can consider the complexity of $c$ from the injection: 
\begin{equation*}
  H_1(M_K;\Z)\xrightarrow{\times c} H_1(V- \Delta;\Z)/torsion.
\end{equation*}
Sometimes we just say a boundary component $M$ of a 4-manifold $W$ has complexity $c$ if 
\begin{equation*}
  H_1(M;\Q)\cong H_1(W;\Q)\cong \Q
\end{equation*}
and the inclusion induces multiplication by $c$ on the free part of first integral homology. Note that $c$ can be always adjusted to be positive, so we assume that $c>0$. Complexity can be introduced on Alexander module and the Blanchfield form, and then the corresponding $\rho$-invariant can be defined. Readers can find this notion in \cite{Cha07}, and also in \cite{COT03} with different terminology ``multiplicity''.

\begin{definition}
The \textit{Alexander module with complexity} $c$ is defined as:
\begin{equation*}
  \A_c (K) = \A(K) \otimes_c \Qt,
\end{equation*}
where the $\Qt$-tensor product $\otimes_c$ is given by the right action of $t\in\Qt$ as multiplication by $t^c$.
\end{definition}

\noindent In the sense that $\A (K)$ is the first homology with rational coefficient of the infinite cyclic cover corresponding $\pi_1(M_K)\xrightarrow{\phi_0} \Z=\langle t \rangle$, we can regard $\A_c (K)$ as the one similarly obtained from the top row composition:

\begin{center}
\begin{tikzcd}
  \pi_1(M_K)\ar{r}{\phi_0}\ar{d} &\Z \ar{r}{t\mapsto t^c} &\Z\\
\pi_1(W)\ar{urr}
\end{tikzcd}
\end{center}
Thus, if $M_K$ bounds $W$ with complexity $c$, then the coefficient system extends to $\pi_1(W)\rightarrow \Z$ as described in above diagram so that the inclusion induces the natural $\Qt$-module homomorphism:
\begin{equation*}
  j_c: \A_c(K)\rightarrow \A(W)=H_1(W;\Qt).
\end{equation*}
Notice that if $\Delta_K(t)$ is the Alexander polynomial, then the order of $\A_c(K)$ is $\Delta_K(t^c)$ and the map $\A(K)\rightarrow \A_c(K)$ sends $t$ to $t^c$.

The Blanchfield form $\Bl_c$ on $\A_c(K)$ is defined similarly as Definition \ref{def: Bl}, but equivalently, can be also defined as a natural extension of the ordinary $\Bl$.
\begin{proposition} \cite{Cha07}
Let $x, y\in \A(K)$ and $f(t), g(t)\in \Qt$. Then the rational Blanchfield form $\Bl_c$ on $\A_c(K)$ extends $\Bl$ in the sense that:
\begin{equation*}
  \Bl_c (x\otimes f(t), y\otimes g(t)) = f(t)\cdot h(\Bl(x, y))\cdot g(t^{-1})
\end{equation*}
where $h:\frac{\Q(t)}{\Qt}\xrightarrow{t\mapsto t^c}\frac{\Q(t)}{\Qt}$.
\end{proposition}

A previously used $\rho$-invariant rational obstruction \cite{Cha07, Kim23} is defined by introducing complexity on the $\rho$-invariant determined by $x\in P\subset \A_c(K)$ in \cite{COT03}. Our obstruction is slightly different than those. We introduce complexity on the first order $L^2$-signatures. However, by the solvability of the quotient group, below vanishing theorem directly comes from \cite[Theorem 4.2] {COT03}. Although these invariants would not be used explicitly by themselves in this paper, except for the case when $P=0$, we contain these for compatibility with above obstructions and better understanding of the usage of localization in later.

\begin{definition}
Let $P$ be an isotropic submodule of $\A_c(K)$ with respect to $\Bl_c$. \textit{A first order $L^2$-signature} $\rho^{(1)}_c(K, P)$ \textit{with complexity} $c$ is defined by $\rho(M_K,\phi)$ where $\phi$ is the quotient map by the kernel of:
\begin{equation*}
  \pi_1(M_K)^{(1)}\rightarrow \frac{\pi_1(M_K)^{(1)}}{\pi_1(M_K)^{(2)}}\rightarrow \A (K)\rightarrow\A_c(K)\rightarrow \A_c(K)/P.
\end{equation*}
\end{definition}

\begin{theorem}\label{CHL-style Q}
If $K$ is rationally slice in a rational homology ball $V^4$ with complexity $c$, then there exists a Lagrangian submodule $P$ of $\A_c (K)$ such that $\rho_c(K, P)=0.$
For a slice disk $\Delta$, above $P$ is given by
\begin{equation*}
  P = \ker (\A_c (K)\rightarrow H_1(V-\Delta;\Qt)).
\end{equation*}
\end{theorem}
\begin{proof}
  Note that the complement $W = V - \nu\Sigma$ is a rational $n$-solution for any positive (half) integer $n$ of the zero surgery $M_K$ (For example, see \cite[Remark 5.5 (6)]{CHL09}). Then $P$ is Lagrangian with respect to $\Bl_c$ \cite[Theorem 5.13]{Cha07}.
 Consider the induced map $j_c:\A_c(K)\rightarrow \A(W)$. Let $P = \ker{j_c}$. Then the following diagram commutes:
\begin{center}
\begin{tikzcd}
  &\frac{\pi_1(M)^{(1)}}{\pi_1(M)^{(2)}} \ar{dd}{j_*} \ar{rr} \ar{dr} & & \A_c(K)/P \ar[hook]{dd}{j_c} \\
  & &\frac{\pi_1(M)^{(1)}}{\ker\phi}\ar[dashed]{dl} \ar[hook]{ur} &\\
&\frac{\pi_1(W)^{(1)}}{\pi_1(W)^{(2)}} \ar[hook]{rr} & &\A(W)
\end{tikzcd}
\end{center}
The dashed arrow is naturally induced since any $\alpha\in\ker\phi$ is sent to $0\in\A(W)$ so that it goes to $\pi_1(W)^{(2)}$. It follows from injectivity of the dashed map that $\phi$ extends to $\pi_1(W)\rightarrow \frac{\pi_1(W)}{\pi_1(W)^{(2)}}$. This extension argument is proved in Lemma \ref{inj-cond} in detail. Now by \cite[Theorem 4.2]{COT03}, $\rho_c(K, P)$ vanishes.
\end{proof}

\begin{remark}
  Readers may compare Theorem \ref{CHL-style Q} with Theorem \ref{CHL-derivative} and Theorem \ref{CHL-Q}. Theorem \ref{CHL-style Q} is a rational version of Theorem \ref{CHL-derivative}, and Theorem \ref{CHL-Q} also obstructs rational sliceness.
\end{remark}

If the set of all Lagrangian submodules can be controlled as complexity $c$ varies, it might be possible to obstruct with $\rho_c$ for only finitely many $c$. This phenomenon is based on the subgroup property and is remarked in \cite[Section 5.2]{Cha07} for the originally defined von Neumann $\rho$-invariants determined by $x\in\A_c(K)$.

\section{Localized $\rho$-invariant with complexity}\label{s-3}
To control the set of all Lagrangian submodules of $\A_c$, we introduce localization on the rational Alexander module $\A_c$.  Then we define a $\rho$-invariant $\rho^{(1)}_{c, p}$ and prove basic properties. The proof for additivity is rather long, so a reader may skip it. Throughout this section, $K$ denotes a knot in $S^3$.

\begin{definition}
  Let $R_p$ be the $\Qt$-module defined in Section \ref{ss-2.3}. \textit{The localized Alexander module with complexity} is
\begin{equation*}
  \A_{c, p}(K) = \A_c(K)\otimes R_p,
\end{equation*}
where the $\Qt$-tensor product is the usual multiplication.
\end{definition}
\noindent As in Section \ref{ss-2.3}, here we similarly expect that some isotropic submodule of $\A_c(K)$ would be killed since the torsion submodules of order relatively prime with $p$ are annihilated in $\A_{c, p}(K)$.
\begin{definition}
Let $\pi_1(M_K)^{(2)}_{c,p}$ be the kernel of the following map:
\begin{equation*}
  \pi_1(M_K)^{(1)}\rightarrow \frac{\pi_1(M_K)^{(1)}}{\pi_1(M_K)^{(2)}} \hookrightarrow \A(K)\hookrightarrow \A_c(K)\rightarrow \A_{c, p}(K).
\end{equation*}
\textit{A localized $\rho$-invariant of $K$ at $p$ with complexity $c$} is defined as:
\begin{equation*}
  \rho^{(1)}_{c, p}(K) = \rho\left(M_K, \phi:\pi_1(M_K)\rightarrow \frac{\pi_1(M_K)}{\pi_1(M_K)^{(2)}_{c, p}}\right).
\end{equation*}
\end{definition}
\begin{remark}
Since $R_p$ is a free $\Qt$-module, it is flat so that $\otimes_{\Qt}R_p$ is exact.
\end{remark}

\begin{proposition}\label{localization property}
Let $p\in \Qt$, $c$ be a positive integer, and $\Delta_K$ be the Alexander polynomial of $K$.
\begin{itemize}
\item [(a)] If $\Delta_K(t^c)$ and $p(t)$ are coprime, then $\rho^{(1)}_{c,p}(K)=\rho^{(0)}(K)$.
\item [(b)] If $\Delta_K(t^c) = p(t)$, then $\rho^{(1)}_{c,p}(K)=\rho^{(1)}(K)$.
\end{itemize}
\end{proposition}

\begin{proof}\hfill
\begin{itemize}
\item[(a)] Let $f(t)$ be an element of $\A_{c, p}(K)$. Since $\A_c(K)$ is annihilated by $\Delta_K(t^c)$ and $\Delta_K(t^c)$ is invertible in $\A_{c, p}(K)$, 
  \begin{equation*}
    f(t)=f(t)\otimes 1=f(t)\Delta_K(t^c)\otimes\frac{1}{\Delta_K(t^c)}=0.
  \end{equation*}
Thus, $\A_{c, p}(K)=0$, which implies that the kernel $\pi_1(M_K)^{(2)}_{c, p}$ is $\pi_1(M_K)^{(1)}$. Therefore,
\begin{equation*}
  \rho_{c,p}^{(1)}(K)=\rho\left(M_K, \phi_0:\pi_1(M_K)\rightarrow \frac{\pi_1(M_K)}{\pi_1(M_K)^{(1)}}\right)=\rho^{(0)}(K).
\end{equation*}
\item[(b)] Let $g(t)$ be an element of the kernel of $\A_c(K)\rightarrow \A_{c, p}(K)$. Then there exists $h(t)\in \Qt$ such that
  \begin{equation*}
    g(t)\otimes 1 =g(t)h(t)\otimes \frac{1}{h(t)} = 0 \otimes \frac{1}{h(t)} = 0,
  \end{equation*}
where $(h(t), \Delta_K(t^c))=1$. Since $g(t)h(t)=0$ in $\A_c(K)$, $g(t)$ has a factor $\Delta_K(t^c)$. Thus, $g(t)=0$ in $\A_c(K)$, which implies that $\A_{c}(K)$ injects to $\A_{c, p}(K)$. Then $\pi_1(M_K)^{(2)}_{c, p}$ is the same as the kernel $\pi_1(M_K)^{(2)}$ of the quotient $\pi_1(M_K)^{(1)}\rightarrow \frac{\pi_1(M_K)^{(1)}}{\pi_1(M_K)^{(2)}}$. Thus,
\begin{equation*}
  \rho_{c,p}^{(1)}(K)=\rho(M_K, \phi_1:\pi_1(M_K)\rightarrow \frac{\pi_1(M_K)}{\pi_1(M_K)^{(2)}})=\rho^{(1)}(K).
\end{equation*}
\end{itemize}
\end{proof}

\noindent Proposition \ref{localization property} is where the localization plays a role. Each condition on $p$ with respect to the Alexander polynomial in (a) and (b) makes $\phi$ be independent of the complexity, so we can obtain a fixed value as $c$ varies.

We prove several propositions which will be used later. Instead of the $0$-surgery $M_K$, we consider a boundary component $M^3$ of $W^4$ with complexity $c$, and use notations such as $\A(M)$, $\A_c(M)$, and $\A_{c, p}(M)$ to denote various versions of Alexander module of $M$.

\begin{proposition}\label{induced}
Suppose that $M$ is a boundary component of $W$ with complexity $c$. Then the inclusion induces a well-defined homomorphism
\begin{equation*}
  i_c: \frac{\pi_1(M)}{\pi_1(M)^{(2)}_{c, p}}\rightarrow \frac{\pi_1(W)}{\pi_1(W)^{(2)}_p}.
\end{equation*}
\end{proposition}
\begin{proof}
  As mentioned in Section \ref{ss-2.4}, there is a well-defined $\Qt$-homomorphism $j_c:\A_c(M)\rightarrow \A(W)$, and this induces $j_c\otimes id:\A_{c,p}(M)\rightarrow \A_p(W)$. Then the diagram below commutes.
\begin{center}
\begin{tikzcd}
&\frac{\pi_1(M)^{(1)}}{\pi_1(M)^{(2)}}\ar{d} \ar[hook]{r} &\A(M) \ar[hook]{r} &\A_c(M)\ar{d}{j_c} \ar{r} & \A_{c, p}(M) \ar{d}{j_c\otimes id} \\
&\frac{\pi_1(W)^{(1)}}{\pi_1(W)^{(2)}} \ar[hook]{rr} &&\A(W) \ar{r} &\A_p(W)
\end{tikzcd}
\end{center}
Since any $x\in \pi_1(M)^{(2)}_{c,p}$ is $0$ in $\A_{c, p}(M)$, it goes to $0$ in $\A_p(W)$ via $j_c\otimes id$. Thus, the image of $x$ under $\pi_1(M)\rightarrow \pi_1(W)$ is contained in $\pi_1(W)_p^{(2)}$. 
\end{proof}
For the value $\rho(M_K,\phi)$ to be obtained by the signature defect of $W$ as Lemma \ref{rho-sg}, it is necessary for the above induced map to be injective. We provide a sufficient condition for such a map to be injective.

\begin{lemma}\label{inj-cond} For a boundary component $M$ of a 4-manifold $W$ with complexity $c$, if the inclusion induced map $j_c\otimes id: \A_{c, p}(M)\rightarrow \A_p(W)$ is injective, then so is $i_c$:
  \begin{equation*}
    \frac{\pi_1(M)}{\pi_1(M)^{(2)}_{c, p}}\xhookrightarrow{i_c} \frac{\pi_1(W)}{\pi_1(W)^{(2)}_p}.
  \end{equation*}
\end{lemma}
\begin{proof} Consider the following commutative diagram with the exact rows and all vertical maps induced by inclusions.
\begin{center}
\begin{tikzcd}
0 \ar{r} &\frac{\pi_1(M)^{(1)}}{\pi_1(M)^{(2)}_{c, p}}\ar{d}{j_*} \ar{r} &\frac{\pi_1(M)}{\pi_1(M)^{(2)}_{c, p}}\ar{d}{i_c} \ar{r} &\frac{\pi_1(M)}{\pi_1(M)^{(1)}}\ar{d}{k_*} \ar{r} &0\\
0 \ar{r} &\frac{\pi_1(W)^{(1)}}{\pi_1(W)^{(2)}_p}\ar{r} &\frac{\pi_1(W)}{\pi_1(W)^{(2)}_p}\ar{r} &\frac{\pi_1(W)}{\pi_1(W)^{(1)}}\ar{r} &0
\end{tikzcd}
\end{center}
Note that $k_*$ is the map between the first homology with coefficient $\Z$ so that it is injective. Now observe that $j_*$ is injective because the below diagram commutes.
\begin{center}
\begin{tikzcd}
\frac{\pi_1(M)^{(1)}}{\pi_1(M)^{(2)}_{c, p}} \ar{d}{j_*} \ar[hook]{r} & \A_{c, p}(M) \ar[hook]{d}{j_c\otimes id} \\
\frac{\pi_1(W)^{(1)}}{\pi_1(W)^{(2)}_p} \ar[hook]{r} & \A_p(W)
\end{tikzcd}
\end{center}
Since $j_*$ and $k_*$ are injective, so is $i_c$.
\end{proof}

To obstruct a knot from being of finite order, one particularly important aspect of localization is that this value is additive under connected sum. It is not expected in Theorem \ref{CHL-derivative} and Theorem \ref{CHL-style Q}.

\begin{proposition}\label{additivity} $\rho^{(1)}_{c, p}(J \# K) = \rho^{(1)}_{c, p}(J) + \rho^{(1)}_{c, p}(K)$.
\end{proposition}
\begin{proof}
We first construct a cobordism $W$ between $M_{J}\amalg M_{K}$ and $-M_{J \# K}$. Then we show that for $M = M_J$, $M_K$, and $M_{J\# K}$, the induced map $i_c$ given in Proposition \ref{induced} is injective so that we obtain $\rho_{c, p}^{(1)}(J) + \rho_{c, p}^{(1)}(K) - \rho_{c, p}^{(1)}(J\# K)=\sigma^{(2)}\left(W, \psi:\pi_1(W)\rightarrow \frac{\pi_1(W)}{\pi_1(W)^{(2)}_p}\right) - \sigma (W)$ by the Lemma \ref{rho-sg}. See the commutative diagram below.
\begin{center}
\begin{tikzcd}
\pi_1(M) \ar{d} \ar{r}{\phi} & \frac{\pi_1(M)^{(1)}}{\pi_1(M)^{(2)}_{c, p}} \ar[hook]{d}{i_c} \\
\pi_1(W) \ar{r}{\psi} & \frac{\pi_1(W)^{(1)}}{\pi_1(W)^{(2)}_{c, p}}
\end{tikzcd}
\end{center}
To complete the proof, we show that $\sigma^{(2)}(W,\psi)$ and $\sigma (W)$ vanish. After construction, most of the remains are just routines, but the proof is long because our cobordism is obtained through several times of glueing. Readers who are familiar with this kind of homology computation may skip those parts. The schematic picture for $W$ is given in Figure \ref{cobordism}. For any knot $K$, the $0$-surgery for $K$ is denoted by $M_K$, and its meridian is denoted by $\mu_K$. We abuse notation by removing the minus sign for orientation-reversed manifolds on the top boundary of cobordisms until \textit{Step 3}.\\

\textit{Step 1. Construct a cobordism $W$ between $M_J \amalg M_K$ and $M_{J\# K}$.}\\

Let $V$ be the 4-manifold obtained by attaching a 1-handle to $M_J\times 1 \subset M_J\times I$ with its dotted circle as an unknot unlinked with $J$ in the Kirby diagram. Isotope the dotted circle to form the $(c, 1)$-cable pattern $P$. Now let $W_1$ be the 4-manifold obtained by attaching a $0$-framed 2-handle to $V$ along the circle $\alpha$ in the leftmost of Figure \ref{Kirby1}. By Kirby calculus illustrated as Figure \ref{Kirby1}, the top boundary is $M_{J_{c, 1}}$. This cobordism $W_1$ between $M_J$ and $M_{J_{c, 1}}$ is a special case of rational homology cobordisms between the 0-surgery for a knot and the 0-surgery for its satellite with a $\Z[\frac{1}{c}]$-slice pattern constructed in \cite{CFHH13}. As proved there, we have isomorphisms
\begin{align*}
H_1(M_J;\Q)&\cong H_1(W_1;\Q)\cong H_1(M_{J_{c, 1}};\Q),\\
H_2(W_1;\Q)&\cong H_2(M_P;\Q)\oplus H_2(M_J;\Q),
\end{align*}
where the first homology is generated by $c[\mu_{J_{c, 1}}]=[\mu_{J}]$.

\begin{figure}[ht!]
\centering
\includesvg{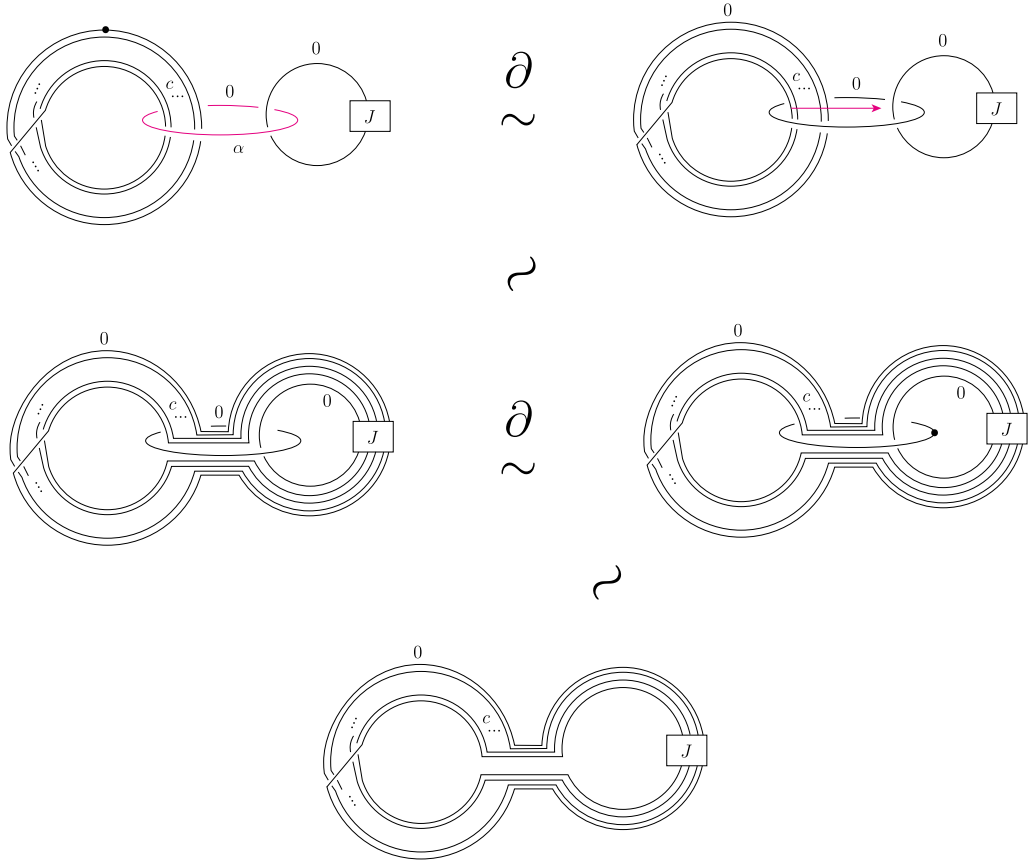}
\caption{Slide $c$ strands highlighted in the top right over the 0-framed 2-handle knotted by $J$ and cancel the 1-handle. The top boundary is $\bdry_+ W_1\cong M_{J_{c, 1}}$.}
\label{Kirby1}
\end{figure}

\begin{figure}[hb!]
\centering
\includesvg{./fig/W2.svg}
\caption{$\bdry_+ W_2\cong M_{(J\# K)_{c, 1}}$.}
\label{W2}
\end{figure}

Take a meridian $\mu_{J_{c, 1}}$ in $M_{J_{c, 1}}=\bdry_+ W_1$ and a meridian $\mu_{K}$ in $M_K\times 1 \subset M_K\times I$, and let $W_2$ be the 4-manifold obtained by glueing $W_1$ and $M_K\times I$ along tubular neighborhoods of a simple closed curve representing $c[\mu_{J_{c, 1}}]$ and $\mu_K$. This connected 4-manifold can be also obtained by attaching a 1-handle to join $W_1$ and $M_K\times I$ and a 0-framed 2-handle along the curve $\beta$ in Figure \ref{W2} to identify $c[\mu_{J_{c, 1}}]$ and $\mu_K$. By a similar Kirby calculus with Figure \ref{Kirby1}, we obtain the top boundary $\bdry_+ W_2$ as $M_{(J\# K)_{c, 1}}$, and hence, $W_2$ is a cobordism between $M_{J_{c, 1}}\amalg M_K$ and $M_{(J\# K)_{c, 1}}$. On the long exact sequence for the pair $(W_2, M_{J_{c, 1}}\amalg M_K)$,
\begin{equation*}
  \cdots \xrightarrow{j_n} H_n(W_2, M_{J_{c, 1}}\amalg M_K;\Q) \xrightarrow{\bdry_n} H_{n-1}(M_{J_{c, 1}}\amalg M_K;\Q) \xrightarrow{i_{n-1}} H_{n-1}(W_2;\Q) \xrightarrow{j_{n-1}}\cdots
\end{equation*}
the boundary map $\bdry_2$ sends the attached 2-handle, which is the generator, to $c[\mu_{J_{c, 1}}]-[\mu_K]$. Thus, we have
\begin{equation*}
  H_1(W_2;\Q)\cong \frac{H_1(M_{J_{c, 1}};\Q)\oplus H_1(M_K;\Q)}{(c[\mu_{J_{c, 1}}]-[\mu_K])} \cong H_1(M_{J_{c, 1}};\Q)\cong H_1(M_J;\Q) \cong H_1(M_K;\Q),
\end{equation*}
which are generated by $c[\mu_{J_{c, 1}}] = [\mu_J]=[\mu_K]$. Moreover, since $\bdry_2$ is injective, $j_2$ is the zero map, and then the inclusion induced map $i_2$ is surjective:
\begin{equation*}
  i_2: H_2(M_{J_{c,1}}\amalg M_K;\Q)\twoheadrightarrow H_2(W_2;\Q).
\end{equation*}

\begin{figure}[htb!]
\centering
\includesvg{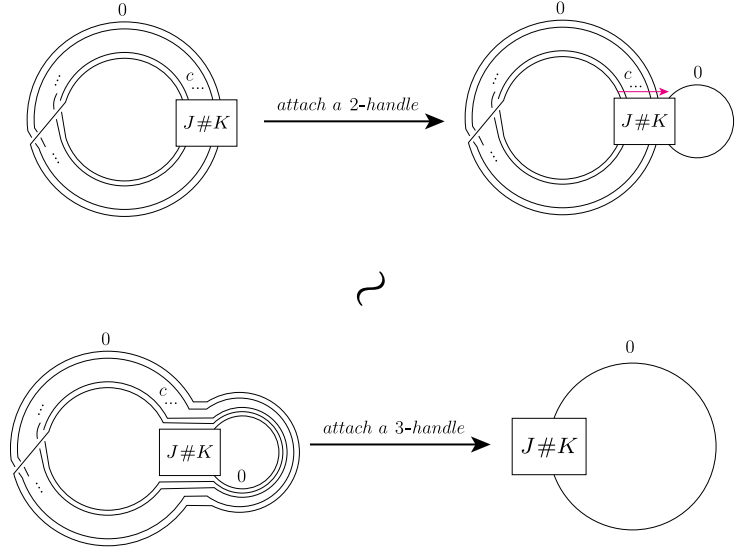}
\caption{After attaching $0$-framed 2-handle, slide $c$ strands of $(J\# K)_{c, 1}$ over the 2-handle to make it unknotted. Canceling the unknotted 2-handle by attaching 3-handle, $\bdry_+ W_3\cong M_{J\# K}$.}
\label{W3}
\end{figure}

Let $W_3$ be the rational homology cobordism from $M_{(J\# K)_{c, 1}}$ to $M_{J\# K}$ constructed in a similar but upside down way with $W_1$. In other words, the attached handles are now a 2-handle and a 3-handle as in Figure \ref{W3}. The Kirby calculus in Figure \ref{Kirby1} and Figure \ref{W3} are explicitly illustrated in \cite{DPR21}.

Glueing $W_3$ to $W_2$ along $M_{(J\#K)_{c,1}}$, the resulting manifold $W$ is a cobordism between  $M_J\amalg M_K$ and $M_{J\# K}$. By the Meyer-Vietoris sequence below,
\begin{equation*}
  \cdots \xrightarrow{j_n} H_n(W;\Q) \xrightarrow{\bdry_n} H_{n-1}(M_{(J\# K)_{c, 1}};\Q) \xrightarrow{i_{n-1}} H_{n-1}(W_2;\Q)\oplus H_{n-1}(W_3;\Q) \xrightarrow{j_{n-1}} \cdots.
\end{equation*}
We have an isomorphism
\begin{equation*}
  H_1(W;\Q)\cong \frac{H_1(W_2;\Q)\oplus H_1(W_3;\Q)}{([\mu_{(J\#K)_{c, 1}}]-c[\mu_{J\#K})]} \cong H_1(W_2;\Q) \cong H_1(W_3;\Q)\cong H_1(M_{J\#K};\Q).
\end{equation*}
The generator is $[\mu_{J\# K}]=c[\mu_{(J\# K)_{c, 1}}]=[\mu_J]=[\mu_K]$. Thus, the inclusion of each boundary component $M_J$, $M_K$, and $M_{J\#K}$ induces isomorphisms on the first rational homology.
Since $i_1$ is injective, the boundary map $\bdry_2$ is the zero map, so $j_2$ is surjective:
\begin{equation*}
  j_2:H_2(W_2;\Q)\oplus H_2(W_3;\Q) \twoheadrightarrow H_2(W;\Q).
\end{equation*}
Since we have seen that $H_2(M_J;\Q)\oplus H_2(M_K;\Q)$ surjects onto $H_2(W_2;\Q)$ and $H_2(M_{J\# K};\Q)$ surjects onto $H_2(W_3;\Q)$, we conclude that every second homology element of $W$ comes from its boundary, namely, the inclusion induced map of boundary is surjective:
\begin{equation*}
  H_2(\bdry W;\Q)\twoheadrightarrow H_2(W;\Q).
\end{equation*}

\begin{figure}[htb!]
\centering
\includesvg{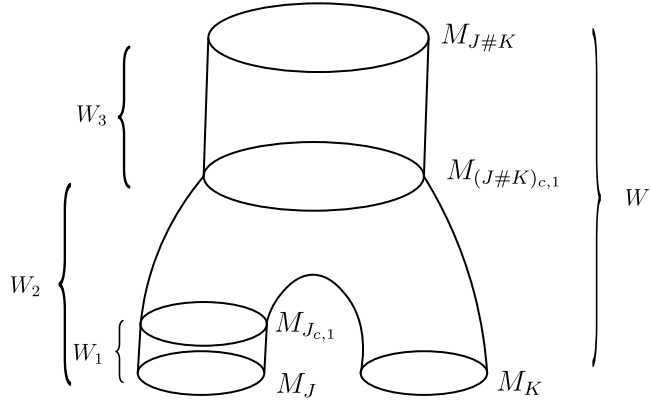}
\caption{A schematic picture of the cobordism $W$ between $M_J \amalg M_K$ and $M_{J\# K}$.}
\label{cobordism}
\end{figure}

\textit{Step 2. Check that $\frac{\pi_1(M)}{\pi_1(M)_{c, p}}\rightarrow \frac{\pi_1(W)}{\pi_1(W)_p}$ is injective for each boundary component $M$ of $W$.}\\

Notice that $H_1(M_K; \Qt) = \A_c (K)$ with complexity $c$, but $H_1(M_{(J\#K)_{c, 1}};\Qt)=\A((J\# K)_{c, 1})$ without complexity, i.e., complexity $c = 1$. In \textit{Step 1}, we have shown that each inclusion for $M=M_J, M_K$, and $M_{J\# K}$ into $W$ induces an isomorphism $H_1(M;\Q)\rightarrow H_1(W;\Q)$. For $\frac{\pi_1(M)}{\pi_1(M)^{(2)}_{c, p}}$ to inject into $\frac{\pi_1(W)}{\pi_1(W)^{(2)}_p}$, by Lemma \ref{inj-cond}, it is enough to check that each Alexander module $\A_{c, p}(M)$ injects to $\A_p(W)$.

We give an another description for $W_1$ in a similar way as we have done for $W_2$. Let $W_0$ be the complement of an unknotted slice disk $\Delta_P$ of the unknotted $(c, 1)$-cable pattern $P$ in $B^4$, which is bounded by the 0-surgery $M_P$. Just as we have identified the simple closed curve representing $c[\mu_{J_{c, 1}}]\subset \bdry_+ W_1$ and $\mu_K\subset M_K\times 1 \subset M_K\times I$, here we take a simple closed curve representing $c[\mu_P]$ in $M_P$ and the meridian $\mu_J$ in $M_J\times 1 \subset M_J\times I$. $W_1$ is the glueing result of $W_0$ and $M_J\times I$ along the tubular neighborhood of such curves $\nu(c\mu_P)=\nu(\mu_{J})\cong S^1\times D^2$. Then we have the following Mayer-Vietoris sequences with coefficient system $\Q[t^{\pm 1}]$:
\begin{equation*}
  \cdots \rightarrow H_n(S^1\times D^2) \rightarrow H_n(W_0)\oplus H_n(M_J\times I) \rightarrow H_n (W_1) \rightarrow \cdots
\end{equation*}
Note that $W_0$ is the 4-ball with a 1-handle attached, so $W_0$ is diffeomorphic to $S^1\times D^3$. Since $S^1\times D^2$ and $W_0$ have contractible infinite cyclic covers as $\R\times D^2$ and $\R\times D^3$ respectively, their homologies with $\Qt$-coefficient are trivial. Thus, 
\begin{equation*}
  H_n(W_1;\Qt) \cong H_n(M_J;\Qt)=\A_c(J).
\end{equation*}

\noindent Similarly, $W_2$ is $W_1 \cup (M_K\times I)$ with intersection $\nu(c\mu_{J_{c, 1}}) = \nu\mu_K\cong S^1\times D^2$. By the exact sequence below,
\begin{equation*}
  \cdots \rightarrow H_n(S^1\times D^2) \rightarrow H_n(W_1)\oplus H_n(M_K\times I) \rightarrow H_n (W_2) \rightarrow \cdots
\end{equation*}
we obtain isomorphisms:
\begin{align*}
H_n(W_2;\Qt) &\cong H_n(W_1;\Qt)\oplus H_n(M_K;\Qt)\\
&\cong H_n(M_J;\Qt)\oplus H_n(M_K;\Qt).
\end{align*}
Furthermore, by Poincare duality and the universal coefficient theorem,
\begin{align*}
0=H_n(W_2, M_J \amalg M_K;\Qt)&\cong H^{4-n}(W_2, M_{(J\# K)_{c, 1}};\Qt)\\
&\cong H_{4-n}(W_2, M_{(J\# K)_{c, 1}}; \Qt).
\end{align*}
Then, by the long exact sequence on the pair $(W_2, M_{(J\# K)_{c, 1}})$, we have
\begin{equation*}
  H_n(M_{(J\# K)_{c, 1}};\Qt)\cong H_n(W_2;\Qt).
\end{equation*}
In particular, the first homology $\A_c(J), \A_c(K)$, and $\A((J\#K)_{c, 1})$ inject into $\A(W_2)$. Since $\otimes R_p$ is exact, the localized Alexander module $\A_{c, p}(J), \A_{c,p}(K)$, and $\A_p((J\#K)_{c, 1})$ inject into $\A_p(W_2)$.

It remains to show that $\A_p(W_2)$ and $\A_{c,p}(M_{J\# K})$ inject into $\A_p(W)$. For the pair $(W, W_2)$, by excision, Poincare duality, and universal coefficient theorem as previous,
\begin{align*}
H_n(W, W_2;\Qt)&\cong H_n(W_3, M_{(J\# K)_{c, 1}};\Qt)\\
&\cong H_{4-n}(W_3, M_{J\# K};\Qt),
\end{align*}
but we can find $H_{4-n}(M_{J\# K};\Qt)\cong H_{4-n}(W_3;\Qt)$ in a similar Meyer-Vietoris sequence with the case for $W_1$ so that $H_n(W, W_2;\Qt)=0$. Thus, $H_n(W_2;\Qt)\cong H_n(W;\Qt)$. In particular, $\A_p(W_2)$ injects into $\A_p(W)$. Moreover, by an excision on the pair $(W, W_3)$, we have
\begin{equation*}
  H_n(W, W_3;\Qt)\cong H_n(W_2, M_{(J\#K)_{c, 1}};\Qt) = 0
\end{equation*}
so that
\begin{equation*}
  H_n(W;\Qt)\cong H_n(W_3; \Qt)\cong H_n(M_{J\# K};\Qt).
\end{equation*}
Therefore, by flatness of $R_p$, the desired condition for injectivity $\A_{c, p}(M) \hookrightarrow \A_p(W)$ for each $M=M_J, M_K$, and $M_{J\# K}$ holds.\\

\textit{Step 3. Compute the signature defect of $W$.}\\

Note that $\bdry W = M_J \amalg M_K \amalg -M_{J\# K}$. Since $\frac{\pi_1(M)}{\pi_1(M)^{(2)}_{c, p}}$ injects into $\frac{\pi_1(W)}{\pi_1(W)^{(2)}_p}$ for each $M=M_J, M_K$, and $-M_{J\# K}$, by Lemma \ref{rho-sg} and Proposition \ref{orientation-reversing}, we have
\begin{equation*}
  \rho_{c, p}^{(1)}(J) + \rho_{c, p}^{(1)}(K) - \rho_{c, p}^{(1)}(J\# K)=\sigma^{(2)}(W, \psi:\pi_1(W)\rightarrow \frac{\pi_1(W)}{\pi_1(W)_p}) - \sigma (W).
\end{equation*}

Let $\K$ be the fraction field of the Ore domain $\Q\left[\frac{\pi_1(W)}{\pi_1(W)_p^{2}}\right]$. Consider the first Meyer-Vietoris sequence in \textit{Step 2} with coefficient $\K$. As in \cite{Coc04},
\begin{equation*}
  H_*(S^1\times D^2;\K)\cong H_*(M_J;\K)\cong H_*(W_0;\K)\cong 0.
\end{equation*}
Thus, $H_*(W_1;\K)=0$, and since $H_*(M_K;\K)=0$ as well, it follows that $H_*(W_2;\K)=0$ from the second sequence in \textit{Step 2}. Since $W_3$ is constructed in the same way with $W_1$, it is also true that $H_*(W_3;\K)=0$ and hence,
\begin{equation*}
  H_2(W;\K)=0.
\end{equation*}
Thus, by Lemma \ref{sg-bound}, the $L^2$-signature $\sigma^{(2)}(W, \psi)$ vanishes. Moreover, while we compute the rational homology of $W$ in \textit{Step 1}, we have already seen that $H_2(W;\Q)$ are supported by the boundary components. Therefore, $\sigma(W)$ vanishes as well, so 
\begin{equation*}
  \rho_{c, p}^{(1)}(J) + \rho_{c, p}^{(1)}(K) - \rho_{c, p}^{(1)}(J\# K)=0.
\end{equation*}

\end{proof}

\begin{proposition}
  Let $-K$ be the mirror image of $K$. Then,
  \begin{equation*}
    \rho_{c, p}^{(1)}(-K) = -\rho_{c, p}^{(1)}(K).
  \end{equation*}
\end{proposition}
\begin{proof}
  Since $S^3_0(-K)$ and $-S^3_0(K)$ are orientation-preserving diffeomorphic, it holds by Proposition \ref{orientation-reversing}.
\end{proof}

\begin{remark}\label{cable 1}
  Consider the $(c, 1)$-cable $K_{c, 1}$ of $K$. Recall the localized $\rho$-invariant $\rho_p$ defined in Section \ref{ss-2.3}. One can easily see that
  \begin{equation*}
    \rho^{(1)}_p(K_{c, 1}) = \rho^{(1)}_{1, p}(K_{c, 1}) = \rho^{(1)}_{c, p}(K).
  \end{equation*}
  The first equality is by definition. The second one can be derived from the cobordism $W_1$ constructed in the proof of Proposition \ref{additivity}, between $M_J$ and $M_{J_{c, 1}}$. See also \cite{CFHH13}.
\end{remark}

\section{Rational Slice Obstruction}\label{s-4}
In this section, we verify that our invariant defined in Section \ref{s-3} gives a rational slice obstruction. To use this, we impose a rather strong condition, called $(c, p)$-anisotropy, for the localized $\A_{c, p}$ to have no nontrivial isotropic submodule with respect to $\Bl_{c, p}$. For such a knot $K$, we show that if $K$ is rationally slice with complexity $c$, then $\rho_{c, p}(K)$ vanishes. Moreover, as Proposition \ref{loc obs}, it can obstruct any connected sum of those knots. After providing several sufficient conditions for $(c, p)$-anisotropy, we prove Theorem \ref{main B}. Throughout this section, $p$ is assume to be a Laurent polynomial with rational coefficients. We also assume that $p$ be symmetric so that the rational Blanchfield form $\Bl_c$ naturally extends to $\Bl_{c,p}$ on the localized Alexander module.

\begin{definition}\label{def: (c,p)-anisotropy} A knot $K$ is \textit{$(c, p)$-anisotropic} if $\A_{c,p}(K)$ has no nontrivial isotropic submodule with respect to the Blanchfield form $\Bl_{c,p}$.
\end{definition}

\begin{theorem}\label{main obstruction}
  Let $K_1, \cdots, K_n$ be $(c, p)$-anisotropic knots. If $K=K_1\#\cdots \# K_n$ is rationally slice with complexity $c$, then $\rho_{c,p}(K)=0$.
\end {theorem}

\begin{proof}
Suppose that $K=K_1 \#\cdots \# K_n$ with $(c,p)$-anisopropic $K_i$, possibly $K_i=\pm K_j$. Let $M_i = M_{K_i}$. Here we similarly go on with the proof of Proposition \ref{additivity} by constructing a 4-manifold $X$ bounded by $M_1,\cdots M_n$ satisfying certain properties. More precisely, the resulting $X$ is required to satisfy the condition in the Lemma \ref{inj-cond} so that $\rho_{c, p}^{(1)}(K)=\sigma^{(2)}(X, \psi)-\sigma(X)$. After that, it is easily verified that the signature defect for $X$ vanishes.\\

\textit{Step 1. Construct a 4-manifold $X$ bounded by $M_1,\cdots, M_n$.}\\

Take $M_i\times I$ and let $\mu_i$ be a meridian of $M_i\times 1\subset M_i\times I$ for each $i$. By identifying all tubular neighborhoods of $\mu_i$, we obtain a cobordism $W$ between $M_1\amalg \cdots \amalg M_n$ and $M_K$. This can be also described as the result attaching $n-1$ copies of a tube $S^1\times B^2 \times I$, say $V_i \times I$ along $V_i \times 0 = \nu\mu_i$ and $V_{i} \times 1 = \nu\mu_{i+1}$. Using a simple Mayer-Vietoris argument, one can check that the inclusion of each boundary component induces an isomorphism. See \cite[Theorem 4.1]{Dav12b} for the detail.
\begin{equation*}
  H_1(M_i;\Q)\cong H_1(W;\Q)\cong H_1(M_K;\Q).
\end{equation*}
which is generated by $[\mu_1]=\cdots =[\mu_n]=[\mu_K]$. We also obtain the following surjection:
\begin{equation*}
  H_2(\amalg_{i=1}^n M_i;\Q)\oplus H_2(M_K;\Q)\twoheadrightarrow H_2(W;\Q).
\end{equation*}

Let $E$ be a slice disk complement for $K$ in a rational 4-ball. Then $H_*(E;\Q)\cong H_*(S^1;\Q)$ and $H_1(E;\Q)\cong H_1(M_K;\Q)$ generated by $c[\mu_K]$. Glueing $W$ and $E$ along $M_K$, let $X$ be the resulting 4-manifold with $\bdry X=M_1\amalg \cdots \amalg M_n$. Then $X$ satisfies the required conditions for injectivity on the first and surjectivity on the second homology with $\Q$ coefficient by the Meyer-Vietoris sequence for $X = W \cup E$.\\

\textit{Step 2. Show that $\frac{\pi_1(M_i)}{\pi_1(M_i)_{c, p}}\rightarrow \frac{\pi_1(W)}{\pi_1(W)_p}$ is injective.}\\

It is enough to show that the localized Alexander module $\A_{c, p}(K_i)$ injects into $\A_p(X)$ by inclusion induced map. It comes from the following lemma.

\begin{lemma}\label{ker=isotropic}
The kernel of the inclusion induced map $\A_{c, p}(K_i)\rightarrow \A_p(X)$ is isotropic with respect to the localized Blanchfield form $\Bl_{c, p}$ on $\bdry X$.
\end{lemma}
\begin{proof}
Note that every isomorphism before localization is preserved since $\otimes R_p$ is exact. Consider the following commutative diagram. Each map is the inclusion induced map.
\begin{center}
\begin{tikzcd}
  \A_{c, p}(K_i) \ar{d}[swap]{\cong} \ar{r}{i_*\otimes id} & \A_p(X)\\
\A_p(W)\\
\A_{c, p}(K) \ar{u}{\cong} \ar{r}{j_*\otimes id} & \A_p(E) \ar{uu}[swap]{\cong}
\end{tikzcd}
\end{center}
Via the vertical isomorphisms, the statement is equivalent with that the kernel of $j_*\otimes id$ is isotropic with respect to $\Bl_{c, p}$ on $M_K$. Let $f(t), g(t)$ be elements in that kernel. Then $f(t)=\frac{f_1(t)}{f_2(t)}$ and $g(t)=\frac{g_1(t)}{g_2(t)}$ where $f_1(t), g_1(t)\in \A_c(K)$ and each of $f_2(t), g_2(t)\in\Qt$ is coprime with $p$. Since $\ker(j_*\otimes id) = (\ker j_*)\otimes {R_p}$, $f_1(t)$ and $g_1(t)$ should lie in the kernel $P = \ker j_*$ so that $\Bl_c(f_1(t), g_1(t))=0$.

By \cite{COT03}, since $E$ is a rational $(1.5)$-solution, $P$ is isotropic with respect to the Blanchfield form $\Bl_c$ on $\A_c(K)$. Choose $k(t)\in \Qt$ such that it represents
\begin{equation*}
  \Bl_c(f_1(t), g_1(t))=[k(t)]=0\in \frac{\Q(t)}{\Qt}.
\end{equation*}
Now we compute the localized Blanchfield form $\Bl_{c, p}(f(t), g(t))$ on $\A_{c, p}(K)$:
\begin{align*}
  \Bl_{c,p}(f(t),g(t))&=\frac{1}{f_2(t)}\Bl_c(f_1(t), g_1(t))\frac{1}{\overline{g_2}(t)}\\
&=\left[ \frac{k(t)}{f_2(t)g_2(t^{-1})} \right] = 0.
\end{align*}
The last equality comes from the fact that $\frac{k(t)}{f_2(t)g_2(t^{-1})}\in R_p$.
\end{proof}
If the kernel of $\A_{c,p}(K_i)\rightarrow \A_p(X)$ were nontrivial, then, since it is isotropic by Lemma \ref{ker=isotropic}, it contradicts with the $(c, p)$-anisotropy assumption for each $K_i$. Thus, each $\A_{c, p}(K_i)$ successfully injects into $\A_p(X)$. Now by Lemma \ref{inj-cond}, $\frac{\pi_1(M_i)}{\pi_1(M_i)^{(2)}_{c, p}}$ injects into $\frac{\pi_1(X)}{\pi_1(X)^{(2)}_p}$. By Proposition \ref{additivity} (additivity) and Lemma \ref{rho-sg}, the invariant is obtained by the signature defect of $X$ as:
\begin{align*}
\rho^{(2)}_{c, p}(K)&= \rho^{(2)}_{c, p}(K_1)+\cdots +\rho^{(2)}(K_n)\\
&=\sigma^{(2)}\left(X, \psi:\pi_1(X)\rightarrow \frac{\pi_1(X)}{\pi_1(X)_p^{(1)}}\right)-\sigma (X).
\end{align*}\\

\textit{Step 3. Compute the signature defect of $X$.}\\

As \cite[Lemma 4.9]{Dav12b}, $H_*(X;\K)=0$, where $\K$ is the fraction field of the Ore domain $\Q\left[\frac{\pi_1(X)}{\pi_1(X)_p^{2}}\right]$. Thus, by Lemma \ref{sg-bound}, the $L^2$-signature $\sigma^{(2)}(X, \psi)$ vanishes. Since $H_2(\bdry X;\Q)\rightarrow H_2(X;\Q)$ is surjective, the ordinary signature $\sigma(X)$ vanishes as well.
\end{proof}

\begin{remark}\label{cable 2}
  Recall $\rho^{(1)}_{1, p}(K_{c, 1}) = \rho^{(1)}_{c, p}(K)$ as in Remark \ref{cable 1}. Thus, if a $(c, p)$-anisotropic knot $K$ has $\rho^{(1)}_{c, p}(K)\neq 0$, then it also obstructs $K_{c, 1}$ from being \textit{strongly rationally slice}, namely rationally slice with complexity $1$. This reminds us of the fact that $K$ is rationally slice with complexity $c$ if and only if $K_{c, 1}$ is strongly rationally slice \cite{CFHH13}.
\end{remark}

Now we provide a sufficient condition for $K$ to be $(c, p)$-anisotropic.
\begin{proposition}\label{c,p-anisotropy condition}
Let $K$ be a knot and $\Delta_K$ be its Alexander polynomial. If each factor of $p$ is symmetric and divides $\Delta_K(t^c)$ with multiplicity at most $1$, then $K$ is $(c, p)$-anisotropic.
\end{proposition}
\begin{proof}
Since $(\A_c(K), \Bl_c)$ is equivalent with $(\A(K_{c, 1}), \Bl)$, the $(c, p)$-anisotropy condition of $K$ is equivalent with the $p$-anisotropy condition of $K_{c, 1}$. Also note that $\Delta_K(t^c)=\Delta_{K_{c, 1}}(t)$. Thus, the result follows from the fact that $p$-anisotropy of $K_{c, 1}$ holds by the Lemma \ref{p-anisotropy condition}.
\end{proof}

\noindent Then the followings are clear.
\begin{proposition}\label{c,p-anisotropy condition 2}
Suppose that each factor of $p$ is symmetric. $K$ is $(c,p)$-anisotropic if one of the following holds:
\begin{itemize}
\item [(a)] $\Delta_K(t^c)$ is squarefree.
\item [(b)] $\Delta_K(t^c)$ and $p$ are coprime.
\end{itemize}
\end{proposition}

\noindent A stronger but simpler condition is irreducibility for $\Delta_K(t^c)$, which follows from Proposition \ref{c,p-anisotropy condition 2}(a). We say a polynomial $p(t)$ is \textit{strongly irreducible} if $p(t^c)$ is irreducible for all positive integers $c$. $p(t)$ and $q(t)$ are said to be \textit{strongly coprime} if $p(t^c)$ and $q(t^d)$ are coprime for all positive integers $c$, $d$. 
%

\begin{proposition}\label{strong irr=>(c,p)-aniso}
If $\Delta_K$ is strongly irreducible, then $K$ is $(c, \Delta_K(t^c))$-anisotropic for all $c$.
\end{proposition}

Now we are ready to prove the obstructive Theorem \ref{main B} which is stated without $(c, p)$-anisotropy condition.
\begin{reptheorem}{main B}
  Let $\{K_i\}$ be a set of knots whose order is finite in $\AC_\Q$ and $\Delta_i$ be the Alexander polynomial of $K_i$. If $\Delta_i$'s are strongly irreducible and pairwisely strongly coprime, and $\rho^{(1)}(K_i)\neq 0$, then $\{K_i\}$ is linearly independent in $\C_\Q$.
\end{reptheorem}

\begin{proof}
Since $K_i$ has finite order in $\AC_\Q$, the Levine-Tristram signature function $\sigma_{K_i}(\omega) = 0$ by \cite{CO93}, \cite{CK02}. Thus, by Theorem \ref{thm: zeroth}, the zeroth order signature vanishes:
\begin{equation*}
  \rho^{(0)}(K_i) = 0.
\end{equation*}

Fix any index $j$, and let $p(t) = \Delta_j(t^c)$. Note that $K_i$ is $(c, p)$-anisotropic for either $i=j$ or $i\neq j$ by Proposition \ref{c,p-anisotropy condition 2}(b) and Proposition \ref{strong irr=>(c,p)-aniso}. By Proposition \ref{localization property}, 
\begin{align*}
\rho_{c, p}(K_i)&=\rho^{(0)}(K_i) = 0 \text{ if }i \neq j. \\
\rho_{c, p}(K_i)&=\rho^{(1)}(K_i) \neq 0 \text{ if }i = j.
\end{align*}

Given $J_1, \cdots, J_n$ arbitrarily chosen in $\{K_i\}$, suppose that $a_1 J_1 \#\cdots \# a_n J_n$ is rationally slice with some complexity $c$. Let $p(t) = \Delta_{k}(t^c)$ for a fixed $k\in\{1,\cdots, n\}$. Then by Theorem \ref{main obstruction}, Proposition \ref{additivity} and the above result for $\rho_{c, p}$,
\begin{align*}
\rho_{c, p}^{(1)}(a_1 J_1 \#\cdots a_n J_n)&=a_1\rho_{c, p}^{(1)}(J_1) + \cdots a_n \rho_{c,p}^{(1)}(J_n)\\
&=\sum\limits_{i\neq k}a_i\rho^{(0)}(J_i)+a_k\rho^{(1)}(J_k)=0.
\end{align*}
Since $\rho^{(0)}(K_i)=0$ and $\rho^{(1)}(K_i)\neq 0$ for all $i$, we have $a_k=0$. In this way, $a_1=\cdots = a_n=0$.
\end{proof}

On the other hand, below corollary immediately follows from Remark \ref{cable 2} in the exactly same way.
\begin{corollary}
  Suppose $\{K_i\}_{i\in I}$ satisfies all conditions in above Theorem \ref{main B}. Then $\{(K_i)_{c, 1}:i \in I, c\in \Z^+\}$ is linearly independent in the classical knot concordance group $\C$.
  \label{cable 3}
\end{corollary}

\noindent We remark that Corollary \ref{cable 3} can be restated in slightly stronger way by replacing $\C$ with the \textit{strong rational concordance group}. The latter means the quotient of $\C$ by strongly rationally slice knots.

\section{Examples and Remark on $(c, 1)$-cable}\label{s-5}
Recall that the $n$-twist knot is the positive Whitehead double of the unknot with $n$ twists, illustrated in Figure \ref{twist knot}. In this section, we prove Theorem \ref{main C} by applying our rational slice obstruction Theorem \ref{main B} to $n$-twist knots which are of order 2 in the algebraic rational concordance group $\AC_\Q$. Let $K_n$ be such an $n$-twist knot and $\Delta_n$ be its Alexander polynomial. Since they are algebraically torsions, their zeroth order signatures $\rho^{(0)}(K)$ vanish by Theorem \ref{thm: zeroth}. Davis \cite{Dav12a, Dav12b} computed their first order signature $\rho^{(1)}(K_n)$ except for finitely many (thirty nine) $n$. The remaining thing is just to show that there exist infinitely many such $K_n$'s which satisfy the required condition on $\Delta_n$ in Theorem \ref{main B}.

\begin{theorem}\label{alg order} \cite[Lemma 7.1]{Dav12a} (cf. \cite{Lev69})
$K_n$ is order $2$ in $\AC$ if and only if there exist positive integers $a$ and $b$ such that $n=a^2-a+b^2$ but there does not exist an integer $m$ such that $n=m^2-m$.
\end{theorem}

\begin{proposition}\label{strong irr}\cite{BD12}
\begin{itemize}
\item [(a)] $\Delta_n$ and $\Delta_m$ are strongly coprime for distinct integers $n$ and $m$.
\item [(b)] $\Delta_n$ is strongly irreducible if $n$ is neither $m(m-1)$ nor $m^k$ for $k>1$ for some $m$.
\end{itemize}
\end{proposition}

\noindent Since the order $2$ cases in Theorem \ref{alg order} may include some perfect power cases, we should exclude those to exploit strong irreducibility.
\begin{proposition} \label{infinitely exist}
There are infinitely many positive integer $n$ such that $K_n$ is of order 2 in $\AC$ and $\Delta_n$ is strongly irreducible. In particular, such $K_n$ has order $2$ in $\AC_\Q$.
\end{proposition}
\begin{proof}
Let $k$ be any odd integer and $n = 36k^2-6k+4$. Then $n=6k(6k-1) + 4 \equiv 1$ mod $3$, but any number of the form $m(m-1)$ is congruent to either $0$ or $2$ mod $3$. Thus, $K_n$ has order 2 in $\AC$. Moreover, $n=2(18k^2-3k+2)$ has a factor $2$ with multiplicity $1$ so that it cannot be any perfect power. Thus, $\Delta_n$ is strongly irreducible so that, by generalized Fox-Milnor condition, $K_n$ is not algebraically rationally slice.
\end{proof}
\noindent Now we are ready to prove Theorem \ref{main C}.
\begin{proof}[Proof of Theorem \ref{main C}]
  Choose infinitely many $K_n$ obtained from Proposition \ref{infinitely exist}. They are of order $2$ in $\AC_\Q$. Their Alexander polynomials are strongly irreducible, and pairwisely strongly coprime by Proposition \ref{strong irr}. For such $n$'s with finitely many exceptions, Davis showed that $\rho^{(1)}(K_n)\neq 0$ in \cite{Dav12a, Dav12b}. Thus, the conditions in Theorem \ref{main B} for $\{K_i\}$ to generate $\Z^\infty$ in $\C_\Q$ are satisfied.
\end{proof}

\noindent On the other hand, when $n=m^2$, $K_n$ is of order $2$ in $\AC$ while it is of order $1$ in $\AC_\Q$ \cite{BD12}. Indeed, $\Delta_n(t^2)$ splits into symmetric pairs. When $m > 8$, $\rho^{(1)}(K_n)\neq 0$ as in \cite{Dav12a, Dav12b}. By Theorem \ref{CHL-Q}, such $K_n$ is not rationally slice. Nevertheless, we cannot conclude that it is of infinite order in $\C_\Q$ since $K_n$ for such $n=m^2$ is not $(2, \Delta_n(t^2))$-anisotropic.

\begin{question*} Let $n = m^2$. Is $K_n$ of infinite order in $\C_\Q$? Even when $m \le 8$, is $K_n$ not rationally slice?
\end{question*}

\noindent Both answers for $n = m = 1$ are ``no" because $K_1$ is the figure-eight knot $4_1$, which is rationally slice.

On the other hand, we also obtain linear independence of their cable in the classical knot concordance group $\C$. In particular, we restate Corollary \ref{main D} in a more general way. Its proof is done by Corollary \ref{cable 3}.
\begin{repcorollary}{main D}
  $\{(K_n)_{c, 1}: K_n\text{ is of order }2 \text{ in }\AC_\Q\text{ with }\rho^{(1)}(K_n)\neq 0, c\in \Z_+\}$ generates an infinite rank subgroup in $\C$.
\end{repcorollary}

\noindent Corollary \ref{main D} reminds us of the following question, a stronger version of \cite[Question 3]{Miy94}, asked by Kang-Park \cite{KP22}.
\begin{question*}
  If $K$ is not slice, then does $K_{c, 1}$ have infinite order in $\C$?
\end{question*}

The twist knots $K_n$ in Corollary \ref{main D}, which are not algebraically rationally slice, give a partially positive answer for above question. In particular, such cable $(K_n)_{c, 1}$ is of infinite order in $\C$. For other cases, in particular, the case when $n=1$, it is hard to tell that its cable has infinite order in $\C$ since $K_1 = 4_1$ itself is not of infinite order in $\C$. It turned out that $(4_1)_{c, 1}$ for odd $c > 1$ are linearly independent in smooth category \cite{HKPS22} (see also \cite{KP22}). This case is especially important because those cables are members of infinite order in the kernel of $\psi:\C\rightarrow \C_\Q$. However, there is no known analogue in the topological category. Question \ref{q3} is asked here again with more specific candidates:

\begin{question*}
Is $(4_1)_{\text{odd}, 1}$ of infinite order in $\C$?
\end{question*}

On the other hand, when $c$ is even, $(4_1)_{c, 1}$ is strongly rationally slice \cite{Kaw80}. When $c$ is $2n$ for any odd $n$, it is known that they have infinite order in smooth category \cite{DKMPS22, KPT24} (see also \cite{ACMPS23}). In particular, they are examples of strongly rationally slice but not smoothly slice. It still remains unknown whether they are topologically slice or not. Even if $\rho$ can be used to obstruct a knot from being topologically slice, our obstruction is far from being able to obstruct topological sliceness of $(4_1)_{c, 1}$.

\begin{question*}
Does there exist a strongly rationally slice but not topologically slice knot?
\end{question*}

\section{Remark on BPH-slice knots}\label{s-6}
In this section, we consider smooth category. A knot $K$ is called \textit{H-slice} in a closed, smooth, orientable 4-manifold $W$ if $K\subset S^3=\bdry(W-B^4)$ bounds a smoothly embedded nullhomologous disk in $W-B^4$. $K$ is said to be $0$\textit{-positive (negative, resp)} if it is H-slice in a simply-connected positive (negative, resp) definite $W^+$ ($W^-$, resp). The following facts are well known:
\begin{theorem} \cite{CHH13}
If $K$ is 0-positive, then
\begin{itemize}
\item [(a)] the signature $\sigma (K) \le 0$,
\item [(b)] the Levine-Tristram signature function $\sigma_K(\omega) \le 0$,
\item [(c)] Ozsváth-Szabó's $\tau(K) \ge 0$.
\item [(d)] Ozsváth-Szabó's correction term $d(S^3_{p/q}(K), i)\le d(L(p, q), i)$.
\end{itemize}
If $K$ is 0-negative, then the inequalities above are given in the opposite way.
\end{theorem}
A knot $K$ is called \textit{$0$-bipolar} if it is both $0$-positive and $0$-negative. In such a particular case, Hom's $\epsilon(K) = 0$ \cite{CHH13}. Since $d(S^3_{p/q}(K), i) = d(L(p, q), i)$ for all coprime integers $p$ and $q$, $\nu^+(K)=0$ by \cite[Theorem 2.5]{NW15} and \cite[Proposition 2.3]{HW16}. It follows from $\nu^+(K)=0$ that $\Upsilon_K(t)=0$ by \cite[Proof of Theorem 1.1]{OSS17}. In summary, if $K$ is $0$-bipolar, then we have (see also \cite{CK21}):
\begin{equation*}
  \sigma(K) = \sigma_K(\omega) = \tau(K) = \epsilon (K) = \Upsilon_K(t)= \nu^+(K)=0.
\end{equation*}

On the other hand, a $0$-bipolar knot is called \textit{biprojective H-slice (or, BPH-slice)} if definite 4-manifolds $W^{\pm}$ in which $K$ is H-slice are specified as $n\mathbb{CP}^2$ and $n\overline{\mathbb{CP}^2}$. Recall the Rasmussen's $s$-invariant \cite{Ras10}. Recently it has turned out that:
\begin{theorem} \cite{MMSW23}
If $K$ is H-slice in $n\mathbb{CP}^2$ ($n\overline{\mathbb{CP}^2}$, resp) for some $n$, then $s(K)\ge 0$ ($\le 0$, resp). In particular, if $K$ is BPH-slice, then
\begin{equation*}
  s(K) = 0.
\end{equation*}
\end{theorem}
\noindent By works of Donaldson \cite{Don83} and Freedman \cite{Fre82}, every simply-connected positive definite smooth 4-manifold is homeomorphic to $n\mathbb{CP}^2$ while it is not known if there exists an exotic $n\mathbb{CP}^2$ or $n\overline{\mathbb{CP}^2}$. Therefore, it is natural to ask the following questions:
\begin{question*}\label{BPH question 1} Does there exist a 0-bipolar but not BPH-slice knot?
\end{question*}
\noindent If such a knot exists, then it implies that there exists a simply-connected closed exotic definite 4-manifold. For existence of non-simply-connected case, see \cite{LLP23, SS23}. The notion of H-slice and BPH-slice has been introduced recently. H-slice knots can be used to detect an exotic indefinite 4-manifolds in gauge theoretic way \cite{MMP24}, and BPH-slice knots are used to construct a potential exotic $n\mathbb{CP}^2$, including $S^4$ for the case when $n=0$ \cite{MP23}. Such homotopy $n\mathbb{CP}^2$'s are shown to be standard \cite{Nak23}.
 
It is interesting to compare the set of rationally slice knots and the set of BPH slice knots. We make a remark that every known examples of rationally slice knots are BPH-slice. On the other hand, not all BPH-slice knots are rationally slice knots (see \cite{CHH13}). As a corollary of Theorem \ref{main C}, we can even take twist knots as such examples:
\begin{corollary}
There are infinitely many BPH-slice twist knots that generate $\Z^\infty$ subgroup of $\C_\Q$.
\end{corollary}
Note that if $K$ is rationally slice, then:
\begin{equation*}
  \sigma(K) = \sigma_K(\omega) = \tau(K) = \epsilon (K) = \Upsilon_K(t)= \nu^+(K) = 0.
\end{equation*}
\noindent On the other hand, it is not known if the Rasmussen's $s$-invariant vanishes for rationally slice knots. We end this article with the following question. 
\begin{question*}
Does there exist a rationally slice but not BPH-slice knot?
\end{question*}

\bibliographystyle{amsalpha}
\bibliography{ref}

\end{document}